\documentclass[12pt]{amsart}
\usepackage[headings]{fullpage}
\usepackage{epsf,amssymb,epic,eepic,epsfig,amsbsy,amsmath,amsfonts}
\numberwithin{equation}{section}
                        \theoremstyle{plain}

\newtheorem{theorem}{Theorem}[section]
\newtheorem{lemma}[theorem]{Lemma}

\newtheorem{proposition}[theorem]{Proposition}

\newtheorem{conjecture}{Conjecture}
\newtheorem{Thm}{Theorem}

\def\BZ{\mathbb Z}
\def\BR{\mathbb R}
\def\BN{\mathbb N}
\def\BQ{\mathbb Q}
\def\BC{\mathbb C}
\def\cE{{\mathcal E}}

\def\la{\langle}
\def\ra{\rangle}

\newcommand{\Vol}{\operatorname{Vol}}

\begin{document}
\title[On the volume conjecture for cables of knots]
{On the volume conjecture for cables of knots}

\author[Thang  T. Q. Le]{Thang  T. Q. Le}
\address{School of Mathematics, 686 Cherry Street,
 Georgia Tech, Atlanta, GA 30332, USA}
\email{letu@math.gatech.edu}

\author[Anh T. Tran]{Anh T. Tran}
\address{School of Mathematics, 686 Cherry Street,
 Georgia Tech, Atlanta, GA 30332, USA}
\email{tran@math.gatech.edu}

\thanks{T.L. was supported in part by National Science Foundation. \\
1991 {\em Mathematics Classification.} Primary 57N10. Secondary 57M25.}

\date{April 18, 2009}

\begin{abstract}
We establish the volume conjecture for $(m,2)$-cables of the figure 8 knot, when $m$ is odd. For $(m,2)$-cables of general knots where $m$ is even, we show that the
limit in the volume conjecture depends on the parity of the color (of the Kashaev invariant). There are many cases when the volume conjecture
for cables of the figure 8 knot is false if one considers all the colors, but holds true if one restricts the colors to a subset of the set of positive integers.
\end{abstract}

\maketitle


\section*{Introduction}

\subsection{The colored Jones polynomial and the Kashaev invariant of a link} Suppose $K$ is framed oriented link with $m$ ordered components in $S^3$. To every $m$-tuple $(n_1,\dots,n_m)$ of positive integers
one can associate a Laurent polynomial
$
J_K(n_1,\dots,n_m;q) \in \BZ[q^{\pm 1/4}]
$, called the colored Jones polynomial, with $n_j$ being the color of the $j$-component of $K$. The polynomial $J_K(n_1,\dots,n_m;q)$ is the quantum link invariant, as defined by Reshetikhin and Turaev \cite{RT,Turaev},
associated to the Lie algebra ${sl}_2(\BC)$, with the color $n_j$ standing for the irreducible ${sl}_2(\BC)$-module $V_{n_j}$ of dimension $n_j$.
Here we use the functorial normalization, i.e. the one for which the colored Jones polynomial of the unknot colored by $n$ is
$$ [n] := \frac{q^{n/2} - q^{-n/2}}{q^{1/2} - q^{-1/2}}.$$

When all the colors are 2, the colored Jones polynomial is the
usual Jones polynomial \cite{Jones}. The colored Jones polynomials of higher colors are more or less  the usual Jones polynomials
of cables of the link.

Following  \cite{MM}, we define  the Kashaev invariant of a link $K$ as the sequence $\la K\ra_N$, $N=1,2,\dots$, by
$$ \la K \ra_N := \frac{J_K(N,\dots, N;q)}{[N]}\vert_{q^{1/4} = \exp(\pi i /2N)}.$$

\subsection{The volume conjecture for knots and links} According to Thurston theory, by cutting the link complement $S^3 \setminus K$ along appropriate disjoint tori one gets a collection of pieces, each is either Seifert fibered or hyperbolic; and $\Vol(K)$ is defined as the sum of the hyperbolic volume of the hyperbolic pieces. It is known that $\Vol(K)= v_3\,  ||S^3\setminus K||$, where
$v_3$ is the volume of the ideal regular tetrahedron, and $||S^3\setminus K||$ is the Gromov norm.
 We can now formulate the volume conjecture of Kashaev-Murakami-Murakami \cite{Kashaev,MM}:

\begin{conjecture} Suppose $K$ is a knot in $S^3$, then
$$ \lim_{N \to \infty}   \frac{\log |\la K\ra_N|}{ N} = \frac{\Vol(K)}{2 \pi}.
$$
\end{conjecture}
For a survey on the volume conjecture, see \cite{Murakami}. Already in \cite{MM} it was noted that the volume conjecture in the above form cannot be true for
split links, since for split links the Kashaev invariant vanishes. There are a few cases of links (of more than one components) when the volume conjecture had been confirmed: in particular,
 the volume conjecture was established for the Borromean rings \cite{GL}, the Whitehead link \cite{Zh}, and more general, for Whitehead chains \cite{Veen}.

 When the Kashaev invariant vanishes, one might hope to remedy the conjecture by renormalizing the colored Jones polynomial. One of consequences of  the present paper
is that the normalization alone is not good enough, we have also to distinguish between the cases
 $N$ even and $N$ odd.

\subsection{Main results}
For a knot $K$ with framing 0, let $K^{(m,p)}$ be the $(m,p)$-cable of $K$, also
with framing 0, see the precise definition in \S \ref{cable}. Note that if $m$ and $p$ are co-prime, then $K^{(m,p)}$ is again a knot.
The two-component link $K^{(0,2)}$ is called the {\em disconnected cable} of $K$. Note that we always have $\Vol(K^{(m,p)})= \Vol(K)$.

In this paper we study the volume conjecture for cables of a knot $K$. It turns out that the case $N$ even and the case $N$ odd are quite different.

\begin{Thm}  Suppose that $K$ is a knot and $K^{(0,2)}$ the disconnected cable of $K$. Then $\la K^{(0,2)}\ra_N =0$ for every {\em even} $N$. \label{1}
\end{Thm}

The case of odd $N$ is quite different, at least for the figure 8 knot:
\begin{Thm}
 Suppose $\cE$ is the figure 8 knot and $\cE^{(0,2)}$ its disconnected cable. Then the volume conjecture holds true for $\cE^{(0,2)}$ if the colors are restricted to the set
 of odd numbers:
 $$ \lim_{N \to \infty,\, \, N \text{ \rm odd}}   \frac{\log |\la \cE^{(0,2)}\ra_N|}{ N} = \frac{\Vol(\cE^{(0,2)})}{2 \pi}.
 \label{2}
$$
 \end{Thm}
 Thus for figure 8 knot, the sequence of Kashaev invariant $|\la \cE^{(0,2)}\ra_N|$ grows exponentially if $N \to \infty$ and $N$ odd. While if $N$ is even, then $|\la \cE^{(0,2)}\ra_N|=0$ (for any knot).

 However, when $m\neq 0$, i.e. when the two components of $K^{(m,2)}$ do have non-trivial linking number, the volume conjecture might still hold true
 even for even $N$.  For example, we have the following result.
 \begin{Thm} Suppose $\cE$ is the figure 8 knot and $m \equiv 2 \pmod 4$. Then the volume conjecture holds true for $\cE^{(m,2)}$ if the colors are restricted to the set
 of numbers divisible by $4$:
 $$ \lim_{N \to \infty,\, \, N \equiv 0 \pmod 4}   \frac{\log |\la \cE^{(m,2)}\ra_N|}{ N} = \frac{\Vol(\cE^{(m,2)})}{2 \pi}.$$
 \label{3}
 \end{Thm}

According to the survey \cite{Murakami}, the volume conjecture has been so far established for the following knots:

\begin{itemize}
\item $4_1$ (by Ekholm),

\item $5_2, 6_1, 6_2$ (by Y. Yokota),
\item torus knots (by Kashaev and Tirkkonen
\cite{KT}), and
\item Whitehead doubles of torus knots of type $(2, b)$ (by  Zheng \cite{Zh}).
\end{itemize}
We add to this the following result.

\begin{Thm}Suppose $\cE$ is the figure 8 knot.  Then the volume conjecture holds true for the knot $\cE^{(m,2)}$ for every odd number $m$.
\label{4}
\end{Thm}

Actually, we will prove some generalizations of Theorems \ref{1}-\ref{4}.

\bigskip
{\bf Remark} 1. We arrived at the theorems through the symmetry principle studied in \cite{KM,Le}, although we will not use the symmetry here.
One important tool in our proof is the Habiro expansion of the colored Jones polynomial \cite{Ha}, which has been instrumental in integrality of
the Witten-Reshetikhin-Turaev invariant of 3-manifolds (see \cite{Ha,BL,Le3}) and in the proof of a generalization of the volume conjecture for
small angles \cite{GL}.

2.  The odd colors correspond to the representations of the group $SO(3)$, or representations of $sl_2$ with highest weights in the root lattice.

\subsection{Plan of the paper} In Section \ref{cable} we exactly formulate the more general results that we want to prove. Sections \ref{calculations} contains some
elementary calculations involving the building blocks in the Habiro expansion. Sections \ref{proof11} and \ref{secn03} contain the proof of the main theorems.

\subsection{Acknowledgements}
We would like to thank R. van der Veen for his interest in this work and for correcting some typos in the manuscript of the paper. We also thank the referee for some suggestions. 

\section{Cables, the colored Jones polynomial, and results}\label{cable}

\subsection{Cables, the colored Jones polynomial} Suppose $K$ is a knot with 0 framing and $m,p$ are two integers with $d$ their greatest common divisor. The $(m,p)$-cable $K^{(m,p)}$ of $K$ is the
link consisting of $d$ parallel copies of the  $(m/d,p/d)$-curve on the torus boundary of a tubular neighborhood of $K$. Here an $(m/d,p/d)$-curve is a curve that is homologically
equal to $m/d$ times the meridian  and $p/d$ times the longitude on the torus boundary.
The cable $K^{(m,p)}$ inherits an orientation from $K$, and we assume that each component of $K^{(m,p)}$ has framing 0.

The colored Jones polynomial is a special case of tangle invariants defined using ribbon Hopf algebras and their modules \cite{RT}.  The ribbon Hopf algebra in our case
is the quantized enveloping algebra $U_h(sl_2)$, e.g. \cite{Ohtsuki}. For each positive integer $n$, there is a unique irreducible $U_h(sl_2)$-module $V_n$ of rank $n$.  In \cite{Ohtsuki} our $J_K(n_1,\dots, n_m;q)$ is denoted by $Q^{sl_2; V_{n_1}, \dots, V_{n_m}}(K)$.

The calculation of $J_{K^{(m,2)}}(N;q)$ is standard: one decomposes $V_N \otimes V_N$ into irreducible components
$$V_N \otimes V_N = \oplus_{l=1}^N V_{2l-1}.$$
Since the $R$-matrix commutes with the actions of the quantized algebra, it acts on each component $V_{2l-1}$ as a scalar $\mu_l$ times the identity. The value of $\mu_l$ is well-known:
$$\mu_l = (-1)^{N-l} q^{\frac{1-N^2}{2}} q^{\frac{l(l-1)}{2}}.$$
Hence from the theory of quantum invariants, we have
$$J_{K^{(m,2)}}(N;q) = \sum_{l=1}^N \mu_l^m J_K(2l-1;q).$$

 The symmetry of quantum invariant at roots of unity \cite{KM, Le} prompts us to combine the color $N-j$ with $N+j$.   So we  rewrite the above formula as follows
\begin{equation}
J_{K^{(m,2)}}(N;q) = a_N^m \sum_{j=1-N, N-j+1 \text{ even}}^{N-1} t_{j,N}^m J_K(N+j;q),
\label{eq10}
\end{equation}
where
$$t_{j,N}= i^{N-1-j}\, q^{\frac{(N+j)^2}{8}} \quad   \text{ and }\quad  a_N=q^{(3-4N^2)/8}.$$

\subsection{General knot case and even $m$} Here we relate the Kashaev invariant of $K^{(m,2)}$ and the colored Jones polynomial of $K_{m/2}$, which is the same knot $K$, only
with framing $m/2$. Increasing the framing by 1 has the effect of multiplying the invariant by $q^{(N^2-1)/4}$, hence
$$ J_{K_{p}}(N;q) = q^{p \frac{N^2-1}{4}} J_K(N;q).$$

\begin{theorem}
Suppose one of the following:

(i) $ m\equiv 0 \pmod 4$ and $N$ is even.

(ii) $ m\equiv 2 \pmod 4$ and $ N\equiv 2 \pmod 4$.
Then, with $q^{1/4} = \exp(\pi i /2N)$, one has
$$ \la K^{(m,2)} \ra_N = q^{m/2} (q^{1/2} -q^{-1/2}) \frac{mN}{4}\sum_{j=1}^{N/2} J_{K_{m/2}}(2j-1;q).$$
In particular, if $m=0$ and $N$ is even, then $\la K^{(m,2)} \ra_N=0$.
\label{11}
\end{theorem}

\subsection{Figure 8 knot case} Let $\cE$ be the figure 8 knot.
We will show that the volume conjecture for $\cE^{(m,2)}$ holds true under some restrictions.

For an integer $m$ there are 4 possibilities, we list them here together with the definition of a set $S_m$:

(i) $m$ is odd. Define  $S_m = \BN$,  the set of positive integers.

(ii) $m \equiv 0 \pmod 8$.  Define $S_m= \{ N\in \BN, N \equiv 1 \pmod 2\}$, the set of odd positive integers.

(iii) $m \equiv 2 \pmod 4$. Define $S_m= \{ N\in \BN, N \not \equiv 2 \pmod 4\}$.

(iv)  $m \equiv 4 \pmod 8$. Define $S_m= \emptyset$.

\begin{theorem} Suppose $\cE$ is the figure 8 knot and $m$ is in one of first three cases (i)--(iii) listed above. Then the volume conjecture for $\cE^{(m,2)}$ holds true if
the colors are restricted to the corresponding set $S_m$, i.e. one has
$$   \lim_{N\to \infty, N \in S_m} \frac {\log  |\la \cE^{(m,2)} \ra_N|}{N} = \frac{\Vol(\cE^{(m,2)})}{2 \pi}.$$
\label{thm01}
\end{theorem}

The proof of the theorem will be given in section \ref{secn03}. Theorems \ref{2}, \ref{3}, and \ref{4} are parts of this theorem.
We still don't have any conclusion for the case (iv).

\section{Some elementary calculations}
\label{calculations}

\subsection{Notations and conventions}
\label{secn01}

We will work with the variable $q^{1/4}$. Let $v= q^{1/2}$. We will use the following notations. Here $j,k,l,N$ are integers.
$$ \{j\} := v^j -v^{-j}, \quad [j]:= \{j\}/\{1\}, \quad S(k,l):= \prod_{k \le j \le l} \{j\}, \quad S'(k,l) :=  \prod_{k \le j \le l, ~j \notin \{0,N\}} \{j\}.$$
$$A(j,k) := \{j-k\} \{j+k\}, \quad {\prod_{i \in I}}' a_i := \sum_{i \in I} \prod_{j \in I \setminus \{i\}} a_j, \quad \text{and} \quad t_{j,N} :=  i^{N-1-j}\, q^{\frac{(N+j)^2}{8}}.$$
For $f,g$ and $h$ in $\BZ[v^{\pm 1}]$, the equation $ f \equiv g \pmod {h}$ means that $f-g$ is divisible by $h$.


\subsection{The building blocks $A(N,k)$ and $S(k,l)$} The expressions $A(N,k)$ and $S(k,l)$ will be the building blocks in the Habiro expansion. We will prove here a few simple
facts.
\begin{lemma}  One has
\begin{eqnarray*}
	A(N-j,k) &\equiv & A(N+j,k) + 2 \{2j\} \{N\} \{N-1\}/\{1\} \pmod{\{N\}^2} \quad \text{in  } \BZ[v^{\pm 1}], \\
	\{N-j\}  &=& -\{N+j\} + \{N\} (v^j + v^{-j}),\\
	t_{-j,N}^m &=& t_{j,N}^m  + \frac{mj}{2}\, t_{j,N}^m\, \{N\} + t_{j,N}^m (v^N +1)^2\,  Q, \quad \text{where } Q \in \BQ[v^{\pm 1}].
\end{eqnarray*}

\label{lm01}
\end{lemma}

\begin{proof} The second equality follows directly from definition. For the first one, by noting that $\{-j\}=-\{j\}$ and $A(j,k)=q^j+q^{-j}-q^k-q^{-k}$ we have
\begin{eqnarray*}
	A(N-j,k) - A(N+j,k) &=& -\{2j\} \{2N\} \\
						&=& -\{2j\} \{N\} (v^N+v^{-N})(v-v^{-1})/\{1\} \\
						&=& -\{2j\} \{N\} [\{N\}(v+v^{-1})-2\{N-1\}]/\{1\} \\
						&\equiv & 2 \{2j\} \{N\} \{N-1\}/\{1\} \pmod{\{N\}^2}.
\end{eqnarray*}

To prove the last one, we note that for a positive integer $k$ there is a polynomial $P(a)$, whose coefficients depend on $k$ only, such that $1-a^k=k(1-a)+(1-a)^2P(a)$. Hence if $mj \ge 0$ then
\begin{eqnarray*}
	t_{-j,N}^m - t_{j,N}^m	&=& -t_{j,N}^m [1-(-v^{-N})^{mj}] \\
						&=& -t_{j,N}^m [(1+v^{-N})mj+(1+v^{-N})^2P(v^{-1})] \\
						&=& -t_{j,N}^m [\frac{mj}{2} (v^{-N} (1+v^N)^2 - \{N\})+(1+v^{-N})^2P(v^{-1})] \\
						&=& \frac{mj}{2}\, t_{j,N}^m\, \{N\} + t_{j,N}^m (v^N +1)^2 [-\frac{mj}{2}\, v^{-N} + v^{-2N}P(v^{-1})],	\end{eqnarray*}
which asserts the third equality. Otherwise, i.e. if $mj<0$, we have
\begin{eqnarray*}
	t_{-j,N}^m - t_{j,N}^m	&=& -t_{j,N}^m [1-(-v^{N})^{-mj}] \\
						&=& -t_{j,N}^m [-(1+v^{N})mj+(1+v^{N})^2P(v)] \\
						&=& -t_{j,N}^m [-\frac{mj}{2} (v^{-N} (1+v^N)^2 + \{N\})+(1+v^{N})^2P(v)] \\
						&=& \frac{mj}{2}\, t_{j,N}^m\, \{N\} + t_{j,N}^m (v^N +1)^2 [\frac{mj}{2}\, v^{-N} - P(v)].	\end{eqnarray*}
						This completes the proof of the lemma.
\end{proof}

\begin{lemma} Suppose $v^N=-1$, then
\begin{equation}
\frac{\displaystyle{t_{j,N}^m [N+j] (\prod_{k=1}^l A(N+j,k)) + t_{-j,N}^m [N-j] \prod_{k=1}^l A(N-j,k) }}{[N]} = t_{j,N}^m D(j,l),
\label{eq01}
\end{equation}
where
$$ D(j,l)= (v^j + v^{-j}) \left( (\prod_{k=1}^l A(j,k)) + 2\{j\}^2  {\prod_{1 \le k \le l}}' A(j,k) \right) + \frac{mj}{2}\{j\} \prod_{k=1}^l A(j,k).$$
\end{lemma}

\begin{proof} From lemma \ref{lm01}, we have
	\begin{eqnarray*}
	\prod_{k=1}^l A(N-j,k) &\equiv & \prod_{k=1}^l [A(N+j,k) + 2 \{2j\} \{N\}\{N-1\}/\{1\}] \pmod{\{N\}^2} \\
						   &\equiv & \prod_{k=1}^l A(N+j,k) + \\
						   &+& 2 \{2j\} \{N\} \{N-1\}/\{1\} {\prod_{1 \le k \le l}}' A(N+j,k) \pmod{\{N\}^2}.
	\end{eqnarray*}
This, together with the last two equalities in lemma \ref{lm01}, implies that when $v^N=-1$ the left hand side of \eqref{eq01} is equal to 	
\begin{eqnarray*}
	&-& \frac{mj}{2}\, t_{j,N}^m \{N+j\} (\prod_{k=1}^l A(N+j,k)) +t_{j,N}^m (v^j+v^{-j}) (\prod_{k=1}^l A(N+j,k)) \\
	&-& 2 \{2j\} \{N-1\} / \{1\} {\prod_{1 \le k \le l}}' A(N+j,k) t_{j,N}^m \{N+j\}
\end{eqnarray*}
Hence \eqref{eq01} follows from the facts that $A(N+j,k) = A(j,k)$ and $\{N+j\} = -\{j\}$ if $v^N=-1$.
\end{proof}

\begin{lemma} One has $D(j,l)=D_1(j,l) + D_2(j,l)$, where
    $$ D_1(j,l)= \begin{cases} \left( \frac{mj}{2}   + \frac{v^j + v^{-j}}{\{j\}} + 2 \{2j\} \sum_{k=1}^l \frac{1}{A(j,k)} \right) S(j-l,j+l), \quad & \text{if } l < \min(j,N-j),\\
0 & \text{if } l \ge \min(j,N-j),
             \end{cases} $$
     and
     $$ D_2(j,l) =\begin{cases}   2  S'(j-l,j+l) \quad &\text{if } j \le l <  N-j,\\
 -2 S'(j-l,j+l) \quad &\text{if } N-j\le l < j,\\
  0 \quad &  \text{if } l <\min(j,N-j) \text{ or } l \ge \max(j,N-j).\end{cases}$$
\label{rem01}
\end{lemma}

\begin{proof}
This can be checked easily by direct calculations.
\end{proof}

\begin{lemma} For $j \le l \le N-j $, the sign of $S'(j-l,j+l)$ is $(-1)^j$.  For $ N-j \le l \le j$, the sign of $S'(j-l,j+l)$ is $(-1)^{N-j}$.
\label{lm02}
\end{lemma}

\begin{proof}
If $j \le l \le N-j $ then the sign of $S'(j-l,j+l)$ is $i^{l+j}(-i)^{l-j}=(-1)^j$. For $ N-j \le l \le j$, the proof is similar.
\end{proof}

\begin{lemma} Suppose $v^N=-1$.
	For $1\le j\le N-1$ and $0\le l \le N-1$ one has
\begin{equation}
 	D(N-j,l) + D(j,l) = \frac{mN}{2} S(j-l,j+l).
 	\label{eq02}
\end{equation}

\end{lemma}

\begin{proof}
Let
$$ B(j,l)=(v^j + v^{-j}) \left( (\prod_{k=1}^l A(j,k)) + 2\{j\}^2  {\prod_{1 \le k \le l}}' A(j,k) \right).$$
Then by definition, $D(j,l)=B(j,l)+\frac{mj}{2}\{j\} \prod_{k=1}^l A(j,k).$ It is easy to see that if $v^N=-1$ then $B(N-j,l)+B(j,l) = 0$ and $$\{N-j\} \prod_{k=1}^l A(N-j,k) = \{j\} \prod_{k=1}^l A(j,k) = S(j-l,j+l).$$
It implies that \eqref{eq02} holds true.
\end{proof}

\section{Proof of Theorem \ref{11}}
\label{proof11}

\subsection{Habiro expansion}
By a deep result of Habiro \cite{Ha}, there are {\em Laurent polynomials} $C_K(l;q) \in \BZ[q^{\pm 1}]$, depending on the knot $K$, such that
$$ J_K(N;q) = [N] \sum_{l=0}^{N-1} C_K(l;q) \prod_{k=1}^l A(N,k).$$
From now on let $q^{1/4}=\exp(\pi i/2N)$. Then $ v^N =-1$. Using Equation \eqref{eq01}, we have for $ 0\le j \le N-1$
$$ \frac{t_{j,N}^m J_K(N+j;q) + t_{-j,N}^m J_K(N-j;q)}{[N]} = \sum_{l=0}^{N-1} C_K(l;q) \, t_{j,N}^m \, D(j,l).$$
Hence Equation \eqref{eq10} implies that
\begin{equation}
\la K^{(m,2)} \ra_N = a_N^m     \sum_{l=0}^{N-1} C_K(l;q) \left( (\sum_{j=1, N-j+1 \text{ even}}^{N-1} t_{j,N}^m\, D(j,l)) + \frac{1-(-1)^N}{4}\,  t_{0,N}^m D(0,l)\right).
\label{eq03}
\end{equation}
\subsection{Proof of Theorem \ref{11} } We assume that $m,N$ satisfy the conditions of Theorem \ref{11}, i.e. $m\equiv 0 \pmod 4$ and $N$ is even; or $ m\equiv 2 \pmod 4$ and $ N\equiv 2 \pmod 4$.

The symmetry \cite{KM,Le} hints that we should combine $j$ and $N-j$.
Since $N$ is even, both $j$ and $N-j$ are odd, so they both appear in the sum \eqref{eq03}. Hence we rewrite the expression in the big parenthesis in right hand side of Equation \eqref{eq03} as follows 
$$(\sum_{j=1, j \text{ odd}}^{N/2-1} t_{j,N}^m D(j,l) + t_{N-j,N}^mD(N-j,l))
+\frac{1-(-1)^{N/2}}{2}\,  t_{N/2,N}^m D(N/2,l).$$
Under the assumption in the theorem on $m$ and $N$, we can easily check that $t_{N-j,N}^m=t_{j,N}^m$. Therefore it follows from Equation \eqref{eq02} that $\la K^{(m,2)} \ra_N$ is equal to $a_N^m\sum_{l=0}^{N-1} C_K(l;q)$ times 
$$(\sum_{j=1, j \text{ odd}}^{N/2-1} t_{j,N}^m \, \frac{mN}{2} S(j-l,j+l)) + \frac{1-(-1)^{N/2}}{2}\ t_{N/2,N}^m \, \frac{mN}{4} S(N/2-l,N/2+l).$$
Now, by noting that $a_N^m=q^{m(3-4N^2)/8}=q^{3m/8}$ and $$t_{N-j,N}^m=t_{j,N}^m=q^{mj^2/8}, \quad S(j-l,j+l)=S(N-j-l,N-j+l),$$ we obtain
\begin{eqnarray*}
\la K^{(m,2)} \ra_N &=&q^{3m/8} \frac{mN}{4}\sum_{l=0}^{N-1} C_K(l;q) \sum_{j=1, j \text{ odd}}^{N-1} q^{mj^2/8} S(j-l,j+l) \\
             &=&q^{m/2} \frac{mN}{4} \sum_{j=1, j \text{ odd}}^{N-1} q^{\frac{m}{2} \cdot \frac{j^2-1}{4}} \{1\} J_K(j;q) \\
           &=& q^{m/2} \{1\} \frac{mN}{4} \sum_{j=1, j \text{ odd}}^{N-1} J_{K_{m/2}}(j;q),
\end{eqnarray*}
where $K_{m/2}$ is $K$ with framing $m/2$. This proves theorem \ref{11}.

\section{Proof of theorem \ref{thm01}}
\label{secn03}

Let $\delta := \exp(\pi i/4)$. We will write $t_j$ for $t_{j,N}$. Then, with $q^{1/4}=\exp(\pi i/2N)$ one has
$$t_{j}=\delta^{3N-2}q^{j^2/8}.$$

For the figure 8 knot $\cE$, we know that $C_\cE(l,q)=1$, see \cite{Ha}. From Equation \eqref{eq03} we have
\begin{eqnarray}
\frac{\la \cE^{(m,2)} \ra_N}{\delta^{(3N-2)m}a_N^m} &=& (\sum_{l=0}^{N-1} \sum_{j=1, N-j+1 \text{ even}}^{N/2-1} q^{mj^2/8}\, D(j,l)) + (\sum_{l=0}^{N-1} \sum_{j=N/2, N-j+1 \text{ even}}^{N-1} q^{mj^2/8} \, D(j,l)) \nonumber\\
&+& \frac{1-(-1)^N}{4}\, \sum_{l=0}^{N-1} D(0,l).
\label{eq12}
\end{eqnarray}

\subsection{The case $j<N/2$}
\label{subsec01} We now consider the first sum in the right hand side of Equation \eqref{eq12}. By lemma \ref{rem01},
$$ D(j,l)= \begin{cases} D_1(j,l)=\left( \frac{mj}{2}   + \frac{v^j + v^{-j}}{\{j\}} + 2 \{2j\} \sum_{k=1}^l \frac{1}{A(j,k)} \right) S(j-l,j+l) & \text{if } l < j\\
D_2(j,l)=2  S'(j-l,j+l)  &\text{if } j \le l <  N-j\\
0 & \text{if } l \ge N-j
             \end{cases} $$
             
             We will consider two subcases when $D(j,l)\neq 0$: $j\le l <N-j$ and $j<l$.

\subsubsection{The subcase $j \le l <N-j $}  By lemma \ref{lm02}, the sign of $S'(j-l,j+l)$ is $(-1)^j=(-1)^{N-1}$. Note that $\{k\}=2i \sin \frac{k\pi}{N}$, hence $S'(j-l,j+l)=(-1)^{N-1}E(j,l)$, where
$$E(j,l)=(\prod_{r=1}^{l-j}2\sin \frac{r\pi}{N})(\prod_{r=1}^{l+j}2\sin \frac{r\pi}{N}).$$
We will see that $E(j,l)$ is maximized when $j=0$ and $l=5N/6.$ Moreover, we have the following result.
\begin{proposition} There exists a nonzero number $C$ such that for any $\alpha \in (\frac{1}{2},\frac{2}{3})$, we have
$$\sum_{j=1, N-j+1 \text{ \rm even}}^{N/2-1} \sum_{l=j}^{N-j} q^{mj^2/8} D_2(j,l) = (-1)^{N-1} C\, E(0, \frac{5N}{6}) N (1+O(N^{3\alpha-2})).$$
\label{prop01}
\end{proposition}

\begin{proof}
By setting
$$s_n=-\sum_{j=1}^n \log | 2\sin \frac{j\pi}{N} |,$$
we have $\log E(j,l)=-s_{l-j}-s_{l+j}.$ Consider the Lobachevsky function
$$L(x):=-\int_0^x \log |2\sin u|du.$$
By a standard argument, see e.g. \cite{GL,Zh}, $s_n=\frac{N}{\pi}L(\frac{n\pi}{N})+O(\log N).$ Hence
$$\log E(j,l)=-\frac{N}{\pi}L(\frac{(l-j)\pi}{N})-\frac{N}{\pi}L(\frac{(l+j)\pi}{N})+O(\log N)=\frac{N}{\pi} f(\frac{j\pi}{N},\frac{l\pi}{N})+O(\log N),$$
where $f(x,y)=-L(-x+y)-L(x+y)$ for $\pi \ge y \ge x \ge 0$ and $\pi \ge x+y.$

It is easy to show that the function $f$ attains its maximum at the unique point $(x,y)=(0,\frac{5\pi}{6})$ . Moreover, the Taylor expansion of $f$ around $(0,\frac{5\pi}{6})$ is
$$f(h,\frac{5\pi}{6}+k)=f(0,\frac{5\pi}{6})-\sqrt{3}(h^2+k^2)+O(|h|^3+|k|^3).$$
 By the same argument as in the proof of theorem 1.2 in \cite{Zh}, there exists $\epsilon >0$ such that
$$  \log E(j,l)\begin{cases}
    	< \log E(0,\frac{5N}{6}) - \epsilon N^{2\alpha-1} + O(1) &\text{  if } j^2+(l-\frac{5N}{6})^2 \ge N^{2\alpha},\\
 		= \log E(0,\frac{5N}{6}) - \frac{\pi\sqrt{3}}{N}[j^2+(l-\frac{5N}{6})^2] + O(N^{3\alpha-2}) &  \text{  otherwise }
  	\end{cases}$$

Let
\begin{eqnarray*}
	I_1 &=& \{(j,l): 1 \le j < N/2, N-j+1 \text{ even}, j \le l \le N-j \text{ and } j^2+(l-\frac{5N}{6})^2 \ge N^{2\alpha}\},\\
	I_2 &=& \{(j,l): 1 \le j < N/2, N-j+1 \text{ even}, j \le l \le N-j \text{ and } j^2+(l-\frac{5N}{6})^2 \le N^{2\alpha}\}.
\end{eqnarray*}
Then we have
$$  E(j,l) = \begin{cases}
    	E(0,\frac{5N}{6}) \exp(-\epsilon N^{2\alpha-1}) O(1) \quad &\text{if } (j,l) \in I_1,\\
 		E(0,\frac{5N}{6}) \exp(-\frac{\pi\sqrt{3}}{N}[j^2+(l-\frac{5N}{6})^2]) (1+O(N^{3\alpha-2}))\quad &  \text{if } (j,l) \in I_2.
  	\end{cases}$$
It implies that
$$\sum_{(j,l) \in I_1} q^{mj^2/8} E(j,l) = E(0,\frac{5N}{6}) N^2 \exp(-\epsilon N^{2\alpha-1}) O(1),$$
and $\sum_{(j,l) \in I_2} q^{mj^2/8} E(j,l)=$
\begin{eqnarray*}	
	&=& \sum_{(j,l) \in I_2} \exp(\frac{\pi i mj^2}{4N}) E(0,\frac{5N}{6}) \exp(-\frac{\pi\sqrt{3}}{N}[j^2+(l-\frac{5N}{6})^2]) (1+O(N^{3\alpha-2})) \\
	&=& \frac{N}{4}\int_{x^2+y^2<N^{2\alpha-1}} \exp(\frac{\pi i m}{4}x^2-\pi\sqrt{3}(x^2+y^2))dxdy E(0,\frac{5N}{6})(1+O(N^{3\alpha-2})) \\
	&=& \frac{N}{4}\int_{\BR ^2} \exp(\frac{\pi i m}{4}x^2-\pi\sqrt{3}(x^2+y^2))dxdy E(0,\frac{5N}{6})(1+O(N^{3\alpha-2})).
\end{eqnarray*}
Hence
\begin{eqnarray*}
 \sum_{j=1, N-j+1 \text{ even}}^{N/2-1} \sum_{l=j}^{N-j} q^{mj^2/8}  D_2(j,l)
&=&2 (-1)^{N-1} \left [\sum_{(j,l) \in I_1} q^{mj^2/8} E(j,l) +\sum_{(j,l) \in I_2} q^{mj^2/8} E(j,l) \right] \\
&=& (-1)^{N-1} C E(0,\frac{5N}{6})N (1+O(N^{3\alpha-2})),
\end{eqnarray*}
where
$$C =\frac{1}{2}\int_{\BR ^2} \exp(\frac{\pi i m}{4}x^2-\pi\sqrt{3}(x^2+y^2))dxdy.$$
We can easily check that $C$ is a nonzero number. This completes the proof the proposition \ref{prop01}.\end{proof}

From now on, we fix the number $\alpha \in (\frac{1}{2},\frac{2}{3})$.

\subsubsection{The subcase $l<j$}  By lemma \ref{lm02}, the sign of $S(j-l,j+l)$ is $i^{l+j}/i^{j-l-1}=i(-1)^l$. Hence we get $S(j-l,j+l)=i(-1)^l F(j,l)$, where
$$F(j,l)= (\prod_{r=1}^{l+j}2\sin \frac{r\pi}{N})/(\prod_{r=1}^{j-l-1}2\sin \frac{r\pi}{N})$$
for $0<l<j<N/2.$ Note that $\log F(j,l)=s_{j-l-1}-s_{l+j}$ and roughly speaking, in this case $F(j,l)$ attains its maximum at $j=N/2$ and $l=N/3$. We claim that
 	
\begin{proposition} One has
\begin{equation}
\sum_{j=1, N-j+1 \text{ \rm even}}^{N/2-1} \sum_{l<j} q^{mj^2/8} D_1(j,l) = N^{3\alpha} F(\frac{N}{2},\frac{N}{3})O(1).
\label{eq13}
\end{equation}
\label{prop10}
\end{proposition}

\begin{proof}
Let
\begin{eqnarray*}
	I_3 &=& \{(j,l): 1 \le j < N/2, N-j+1 \text{ even}, l<j \text{ and } (j-\frac{N}{2})^2+(l-\frac{N}{3})^2 \ge N^{2\alpha}\},\\
	I_4 &=& \{(j,l): 1 \le j < N/2, N-j+1 \text{ even}, l<j \text{ and } (j-\frac{N}{2})^2+(l-\frac{N}{3})^2 \le N^{2\alpha}\}.
\end{eqnarray*}
By the same argument as in the proof of the previous proposition, we have
$$  F(j,l) \begin{cases}
    	= F(\frac{N}{2},\frac{N}{3}) \exp(-\epsilon N^{2\alpha-1}) O(1) \quad &\text{if } (j,l) \in I_3,\\
 		\le F(\frac{N}{2},\frac{N}{3}) \quad &  \text{if } (j,l) \in I_4.
  	\end{cases}$$

To prove the proposition, we need the following three lemmas:
  	
\begin{lemma} One has $$\sum_{(j,l) \in I_3} q^{mj^2/8} \left( \frac{mj}{2} + \frac{v^j + v^{-j}}{\{j\}} + 2 \{2j\} \sum_{k=1}^l \frac{1}{A(j,k)} \right) S(j-l,j+l) =$$
$$=N^5 F(\frac{N}{2},\frac{N}{3}) \exp(-\epsilon N^{2\alpha-1}) O(1).$$
\label{lm03}
\end{lemma}

\begin{proof} It suffices to show that if $(j,l) \in I_3$ then
\begin{equation} \frac{mj}{2}   + \frac{v^j + v^{-j}}{\{j\}} + 2 \{2j\} \sum_{k=1}^l \frac{1}{A(j,k)}=N^3 O(1). \label{eq11} \end{equation}
Indeed, since $\sin x \ge \frac{2}{\pi} x$ for $x \in [0, \frac{\pi}{2}]$, we have
$$|\frac{v^j + v^{-j}}{\{j\}}|=|\frac{\cos (j\pi /N)}{\sin (j\pi /N)}| \le \frac{\pi}{2} \cdot \frac{N}{j\pi} \le \frac{N}{2}.$$
Note that $k \le l \le j-1$, hence $$|A(j,k)|=|\{j-k\}\{j+k\}|=4 \sin \frac{(j-k)\pi}{N} \sin \frac{(j+k)\pi}{N} \ge 4 \sin \frac{\pi}{N} \sin \frac{(2j-1)\pi}{N}.$$
It implies that
$$|\frac {\{2j\}} {A(j,k)}| \le \frac {\sin \frac{2j\pi}{N}} {2\sin \frac{\pi}{N} \sin \frac{(2j-1)\pi}{N}}.$$
If $j > \frac{N}{4}$ then $\sin \frac{(2j-1)\pi}{N}>\sin \frac{2j\pi}{N} >0$, therefore $$|\frac {\{2j\}} {A(j,k)}| \le \frac {1} {2\sin \frac{\pi}{N}} \le \frac{N}{4}.$$
If $j \le \frac{N}{4}$ then
$$|\frac {\{2j\}} {A(j,k)}| \le \frac {1} {2\sin^2 \frac{\pi}{N}} \le \frac{N^2}{8}.$$
Hence \eqref{eq11} holds true and then the claim of the lemma is proved.
\end{proof}

\begin{lemma} One has
$$\sum_{(j,l) \in I_4} q^{mj^2/8} \left( \frac{m}{2} (j-\frac{N}{2})  + \frac{v^j + v^{-j}}{\{j\}} + 2 \{2j\} \sum_{k=1}^l \frac{1}{A(j,k)} \right) S(j-l,j+l)=$$
$$=N^{3\alpha}F(\frac{N}{2},\frac{N}{3})O(1).$$
\label{lm04}
\end{lemma}

\begin{proof}
Since $(j-\frac{N}{2})^2+(l-\frac{N}{3})^2 \le N^{2\alpha}$, we have
$$|\frac{v^j + v^{-j}}{\{j\}}|=\cot \frac {j\pi} {N}=|\frac{\pi}{2}-\frac {j\pi} {N}|O(1)=N^{\alpha-1}O(1),$$
and
$$|\frac{\{2j\}}{A(j,k)}| \le |\frac{\{2j\}}{A(j,l)}| = \frac {\sin \frac{2j\pi}{N}} {2\sin \frac{(j-l)\pi}{N} \sin \frac{(l+j)\pi}{N}} = N^{\alpha -1} O(1),$$
which proves the equality of the lemma.
\end{proof}

\begin{lemma} One has
	$$\sum_{(j,l) \in I_4} q^{mj^2/8} S(j-l,j+l) = N^{3\alpha-1} F(\frac{N}{2}, \frac{N}{3}) O(1).$$
	\label{lm05}
\end{lemma}

\begin{proof} 

Denote by $L$ the left hand side of the equality. We first see that $L$ is essentially equal to 
the following expression
$$L'=\frac{1}{2} \sum_{(j,l) \in I_4} q^{mj^2/8} [S(j-l,j+l) + S(j-l-1,j+l+1)].$$
Note that for $(j,l) \in I_4$, we have $S(j-l,j+l) + S(j-l-1,j+l+1)=$
	\begin{eqnarray*}
		 &=& S(j-l,j+l) [1-4\sin \frac{(j-l-1)\pi}{N} \sin \frac{(j+l+1)\pi}{N}] \\
		&=& S(j-l,j+l) [(2+2\cos \frac{2j\pi}{N})-(1+2\cos \frac{(2l+1)\pi}{N})]\\
		&=& N^{\alpha-1} F(\frac{N}{2},\frac{N}{3})O(1).
	\end{eqnarray*}
This implies that $L'=N^{3\alpha-1} F(\frac{N}{2}, \frac{N}{3}) O(1).$ 
Hence to complete the proof of the lemma, we need to estimate the difference between $L$ and $L'.$ 

Let $J=\{j: (j,l) \in I_4 \text{~for some~} l\}.$ For each $j \in J,$ let $J_j=\{l: (j,l) \in I_4\}.$ 
We have
\begin{equation}
L'-L=\frac{1}{2} \sum_{j \in J} q^{mj^2/8} \sum_{l \in J_j} S(j-l-1,j+l+1)-S(j-l,j+l).
\label{eq30}
\end{equation}
For each $j$ in $J$, it is easy to see that the set $J_j$ is just a closed interval $[a_j,b_j].$ Hence $$\sum_{l \in J_j} S(j-l-1,j+l+1)-S(j-l,j+l)=S(j-b_j-1,j+b_j+1)-S(j-a_j,j+a_j)$$ has absolute value less than or equal to $2F(\frac{N}{2}, \frac{N}{3}).$ Equation \eqref{eq30} then implies that $|L'-L| \le N^{\alpha} 2F(\frac{N}{2}, \frac{N}{3}).$ 
Since $\alpha<3\alpha-1,$ the conclusion of the lemma follows.
\end{proof}

We now come back to the proof of proposition \ref{prop10}. It is easy to see that lemmas \ref{lm04} and \ref{lm05} imply
$$\sum_{(j,l) \in I_4} q^{mj^2/8} \left( \frac{mj}{2}  + \frac{v^j + v^{-j}}{\{j\}} + 2 \{2j\} \sum_{k=1}^l \frac{1}{A(j,k)} \right) S(j-l,j+l)=N^{3\alpha}F(\frac{N}{2},\frac{N}{3})O(1).$$
Hence, by combining this with lemma \ref{lm03}, we obtain the equality of the proposition.
\end{proof}

We can now estimate the first sum in the right hand side of Equation \eqref{eq12}. To do that, we need one more lemma which allows us to express the right hand side of Equation \eqref{eq13}, and hence the sum in its left hand side, in terms of $E(0,\frac{5N}{6}).$

\begin{lemma} One has $$F(\frac{N}{2},\frac{N}{3})=\frac{1}{N}E(0,\frac{5N}{6}) O(1).$$
	
	\label{lem10}
\end{lemma}

\begin{proof}
The follows from the fact that $\prod_{k=1}^{N-1} (2\sin \frac{k\pi}{N}) = N.$
\end{proof}

From proposition \ref{prop10} and lemma \ref{lem10}, we have
$$\sum_{j=1, N-j+1 \text{ even}}^{N/2-1} \sum_{l<j} q^{mj^2/8} D_1(j,l) = N^{3\alpha-1} E(0,\frac{5N}{6})O(1).$$
This, together with proposition \ref{prop01}, implies that
\begin{equation}
\sum_{l=0}^{N-1} \sum_{j=1, N-j+1 \text{ even}}^{N/2-1}  q^{mj^2/8} D(j,l) = (-1)^{N-1} C E(0, \frac{5N}{6}) N (1+O(N^{3\alpha-2})).
\label{eq04}
\end{equation}

\subsection{The case $j \ge N/2$} Again, by lemma \ref{rem01} we have
$$ D(j,l)= \begin{cases} D_1(j,l)=\left( \frac{mj}{2}   + \frac{v^j + v^{-j}}{\{j\}} + 2 \{2j\} \sum_{k=1}^l \frac{1}{A(j,k)} \right) S(j-l,j+l) & \text{if } l < N-j\\
D_2(j,l)=-2  S'(j-l,j+l) &\text{if } N-j \le l <  j\\
0 & \text{if } l \ge j
             \end{cases} $$
In this case all the estimations we have done in section \ref{subsec01} also work for the second sum in the right hand side of Equation \eqref{eq12} except two things. The first one is that for $N-j \le l < j$, by lemma \ref{lm02}, the sign of $S'(j-l,j+l)$ is $(-1)^{N-j}=-1$ and the other is $$q^{mj^2/8}=q^{m(N-j)^2/8} \exp(\pi i m(2j-N)/4),$$
where $$\exp(\pi i m(2j-N)/4) =
			\begin{cases} \beta=\exp(\pi i m(N+2)/4) , \quad & \text{if } N-j+1 \equiv 0 \pmod 4,\\
\gamma=\exp(\pi i m(N-2)/4) , \quad & \text{if } N-j+1 \equiv 2 \pmod 4.
            \end{cases} $$
Therefore by similar arguments as in the proof of Equation \eqref{eq04}, we get
\begin{equation}
\sum_{l=0}^{N-1} \sum_{j=N/2, N-j+1 \text{ even}}^{N-1}  q^{mj^2/8} D(j,l) = \frac{1}{2}(\beta+\gamma)C E(0, \frac{5N}{6}) N (1+O(N^{3\alpha-2})).
\label{eq15}
\end{equation}

\subsection{Proof of theorem \ref{thm01}}

From Equations \eqref{eq04} and \eqref{eq15}, we have
\begin{equation}
\sum_{l=0}^{N-1} \sum_{j=1, N-j+1 \text{ even}}^{N-1} q^{mj^2/8} D(j,l) = \frac{1}{2}(\beta+\gamma+2(-1)^{N-1})C E(0, \frac{5N}{6}) N (1+O(N^{3\alpha-2})).
\label{eq16}
\end{equation}

Moreover, it is easy to see that 
\begin{equation}
\sum_{l=0}^{N-1} D(0,l)= N^\alpha E(0,\frac{5N}{6})O(1).
\label{eq17}
\end{equation} Therefore, to complete the proof of theorem \ref{thm01}, we need the following lemma

\begin{lemma} We have $\beta+\gamma+2(-1)^{N-1}=0$ if and only if one of the following holds:

$(i)$ $ m\equiv 0 \pmod 4$ and $N$ is even.

$(ii)$ $ m\equiv 2 \pmod 4$ and $ N\equiv 2 \pmod 4$.

$(iii)$ $m \equiv 4 \pmod 8$ and $N$ is odd.\\
Moreover, if $\beta+\gamma+2(-1)^{N-1} \not=0$ then $|\beta+\gamma+2(-1)^{N-1}| \ge 2$.
\label{lm06}
\end{lemma}

\begin{proof}
Suppose that $\beta+\gamma+2(-1)^{N-1}=0$. If $N$ is even then $\beta+\gamma=2$. Since $|\beta|=|\gamma|=1$, it implies that $\beta=\gamma=1$ which means that $ m\equiv 0 \pmod 4$; or $ m\equiv 2 \pmod 4$ and $ N\equiv 2 \pmod 4$. If $N$ is odd then similarly, we have $\beta=\gamma=-1$, which is equivalent to the condition that $m \equiv 4 \pmod 8$.

Note that if $m$ is odd then $\beta+\gamma=0$, hence $|\beta+\gamma+2(-1)^{N-1}|=2$. Now let us consider the case $m$ is even. Then $\beta=\gamma$ and $\beta+\gamma+2(-1)^{N-1}=2(\beta+(-1)^{N-1})$, so it is remaining to show that $|\beta+(-1)^{N-1}| \ge 1$. Since
$\beta=i^{\frac{m}{2}(N+2)} \in \{\pm 1, \pm i\}$ and $\beta+(-1)^{N-1} \not= 0$, it implies that the angle between $\beta$ and $(-1)^{N-1}$ is less than or equal to $\frac{\pi}{2}$. Hence  $$|\beta+(-1)^{N-1}|^2 \ge |\beta|^2+|(-1)^{N-1}|^2=2,$$
which completes the proof of the lemma.
\end{proof}

Under the assumption of theorem \ref{thm01}, by lemma \ref{lm06}, $|\beta+\gamma+2(-1)^{N-1}| \ge 2$. Hence from Equations \eqref{eq16} and \eqref{eq17} we get
$$\frac{\la \cE^{(m,2)} \ra_N}{\delta^{(3N-2)m}a_N^m} = \frac{1}{2}(\beta+\gamma+2(-1)^{N-1})C E(0, \frac{5N}{6}) N (1+O(N^{3\alpha-2})),$$
which, together with the simple fact that $$\log E(0,\frac{5N}{6}) = -\frac{2N}{\pi} L(\frac{5\pi}{6}) + O(\log N)=\frac{2N}{\pi} L(\frac{\pi}{6}) + O(\log N),$$ confirms the volume conjecture for $\cE^{(m,2)}$:
$$ 2 \pi \, \lim_{N\to \infty, N \in S_m} \frac {\log  |\la \cE^{(m,2)} \ra_N|}{N} = 4L(\frac{\pi}{6})=\Vol(\cE^{(m,2)}).$$

\end{document}